\documentclass[12pt,reqno]{amsart}
\usepackage{amsmath}
\usepackage{amssymb}
\usepackage{mathrsfs}
\usepackage{amstext}
\usepackage{a4wide}
\usepackage{graphicx}
\usepackage[numbers,sort&compress]{natbib}
\allowdisplaybreaks \numberwithin{equation}{section}
\usepackage{color}
\usepackage{cases}

\numberwithin{equation}{section}

\newtheorem{theorem}{Theorem}[section]
\newtheorem{proposition}[theorem]{Proposition}

\newtheorem{lemma}[theorem]{Lemma}
\newtheorem*{Yudovich's Theorem}{Yudovich's Theorem}

\theoremstyle{definition}

\newtheorem{definition}[theorem]{Definition}

\theoremstyle{remark}
\newtheorem{remark}[theorem]{Remark}

\begin{document}

\title
[Stability of 2D steady Euler flows]{Stability of 2D steady Euler flows related to least energy solutions of the Lane-Emden equation}

 \author{Guodong Wang}
\address{Institute for Advanced Study in Mathematics, Harbin Institute of Technology, Harbin 150001, P.R. China}
\email{wangguodong@hit.edu.cn}


\begin{abstract}
In this paper, we investigate nonlinear stability of planar steady Euler flows related to least energy solutions of the Lane-Emden equation in a smooth bounded domain. We prove the orbital stability of these flows in terms of both  the $L^s$ norm of the vorticity for any $s\in(1,+\infty)$ and the energy norm. As a consequence, nonlinear stability is obtained when the least energy solution is unique, which actually holds for a large class of domains and exponents. The proofs are based on a new variational characterization of least energy solutions in terms of the vorticity, a compactness argument, and proper use of conserved quantities of the Euler equation.
\end{abstract}

\maketitle

\section{Introduction}
\subsection{2D Euler equation and conserved quantities}
Let $D\subset \mathbb R^2$ be a smooth bounded domain. The evolution of an incompressible inviscid fluid in $D$ is governed by the following two-dimensional (2D) Euler equation
\begin{equation}\label{euler}
\begin{cases}
\partial_t\mathbf v+(\mathbf v\cdot\nabla)\mathbf v=-\nabla P& x=(x_1,x_2)\in D, \,t\in(0,+\infty),\\
\nabla\cdot\mathbf v=0& x\in D,\,t\in(0,+\infty),\\
\mathbf v(0,x)=\mathbf v_0(x)&x\in\partial D,
\end{cases}
\end{equation}
where $\mathbf v=(v_1,v_2)$ is the velocity field, $P$ is the scalar pressure, and $\mathbf v_0$ is a given divergence-free field (i.e., $\nabla\mathbf\cdot\mathbf v_0=0$). Here the fluid is assumed to be of unit density. For boundary condition, we assume that there is no mass flow across  the boundary, i.e.,
\begin{align}\label{bc}
\mathbf v\cdot\mathbf n=0\quad  x\in \partial D,\,t\in[0,+\infty),
\end{align}
where $\mathbf n$ denotes the unit normal exterior to $D$. Of course $\mathbf v_0$ is supposed to satisfy \eqref{bc} as well.

Introduce the scalar vorticity $\omega:=\partial_{x_1}v_2-\partial_{x_2}v_1,$ the signed magnitude of the vorticity vector curl$\mathbf v$ (regarding \eqref{euler} as a system of equations in three dimensions with $\mathbf v$ and $P$ not depending on the third spatial variable). Below  we show that the Euler equation \eqref{euler} has an equivalent vorticity form. First, taking the curl on both sides of the first equation of \eqref{euler} we see that $\omega$ satisfies the following transport equation
\begin{equation}\label{vor}
\partial_t\omega+\mathbf v\cdot\nabla\omega=0.\end{equation}
On the other hand, $\mathbf v$ can be determined in terms of $\omega$. In fact, since $\mathbf v$ is divergence-free and $\mathbf v\cdot\mathbf n=0$ on $\partial D$, we can define a function $\psi,$ called the \emph{stream function}, such that  
\begin{equation}\label{sff}
\mathbf v=\nabla^\perp \psi.
\end{equation}
 Here and henceforth, we use $\mathbf b^\perp$ to denote the clockwise rotation  through $\pi/2$ of some planar vector $\mathbf b$, and $\nabla^\perp f$ to denote $(\nabla f)^\perp$ for some function $f$.
It is easy to see that
\begin{equation}\label{pe}
-\Delta \psi=\omega,
\end{equation}
and
\begin{equation}\label{bryc}\psi=\Psi_i \mbox{ on each connected component $\Gamma_i$ of $\partial D$},
\end{equation}
 with $\Psi_i$ being undetermined constants depending on the time variable. To determine $\psi$ in terms of $\omega$, we should take into account Kelvin's circulation theorem, stating that the circulation of the velocity on each $\Gamma_i$ is a conserved quantity. More precisely,
\begin{equation}
\label{cir}\oint_{\Gamma_i}\nabla\psi\cdot\mathbf n dS=c_i
\end{equation}
with $c_i$ being constants depending on $\mathbf v_0$ only. It is not hard to prove that there exists a unique solution $\psi$ to  \eqref{pe}, \eqref{bryc} and \eqref{cir} up to a constant. Therefore $\mathbf v$ is uniquely determined by $\omega$ in view of \eqref{sff}. The way in which $\mathbf v$ is recovered from $\omega$ is usually called the Biot-Savart law.

If $D$ is additionally simply-connected, then the Biot-Savart law has the following simple form
\begin{equation}\label{bs}
\mathbf v=\nabla^{\perp}\mathcal G\omega,
\end{equation}
where $\mathcal G$ is the Green's operator with zero Dirichlet boundary condition, i.e., $\mathcal G\omega$ satisfies
\begin{equation*}
\begin{cases}
-\Delta\mathcal G\omega=\omega &x\in D,\\
\mathcal G\omega=0&x\in\partial D.
\end{cases}
\end{equation*}
 Therefore we have obtained the evolution equation of the vorticity in a simply-connected domain
\begin{equation}\label{ve}
\partial_t\omega+\nabla^{\perp}\mathcal G\omega\cdot\nabla\omega=0,
\end{equation}
which is usually called the \emph{vorticity equation}. Note that if $\omega$ is a solution to the vorticity equation \eqref{ve}, then there is a corresponding pair  $(\mathbf v, P)$  solving the Euler equation \eqref{euler} given by
\[\mathbf v=\nabla^\perp\mathcal G\omega,\quad P(x)=\int_{L_{ x_0, x}}\omega( y)\mathbf v^\perp( y)\cdot d y-\frac{1}{2}|\mathbf v(x)|^2,\]
where $x_0\in D$ is an arbitrarily chosen point and $L_{ x_0, x}$ is an arbitrary  $C^1$ curve joining $x_0$ and $ x$.

In this paper, for the sake of simplicity, we always assume that $D$ is a simply-connected smooth bounded domain. The formulation of the stability problem of steady Euler flows in multi-connected domains is a little different, but has no essential difficulty.

The solvability of the initial value problem for the vorticity equation \eqref{ve} has been studied by many authors in various settings, including Wolibner \cite{Wo} in H\"older space, Yudovich \cite{Y} in $L^\infty$, DiPerna-Majda \cite{DM} in $L^p, 1<p<+\infty$ and Delort \cite{De} in the space of non-negative Radon measures in $H^{-1}.$ In this paper, to make the statements concise, we confine ourselves to the $L^\infty$ setting in which the vorticity equation admits a unique weak solution with fine properties. Below we state one version of Yudovich's result, the detailed proof of which can be found in Burton  \cite{B5} or Marchioro-Pulvirenti  \cite{MPu}.

\begin{Yudovich's Theorem}\label{A}
For any $\omega_0\in L^\infty(D)$, the vorticity equation \eqref{ve} admits a unique weak solution  $\omega(t,x)\in L^\infty((0,+\infty)\times D)$ satisfying
\begin{equation}\label{tweak}
  \int_D\omega_0(x)\phi(0,x)dx+\int_0^{+\infty}\int_D\omega\partial_t\phi+ \omega\nabla^\perp \mathcal G\omega\cdot\nabla\phi dxdt=0
    \end{equation}
for all $\phi\in C_c^{\infty}(\mathbb R\times D)$.
 Moreover, this weak solution has the following three properties:
\begin{itemize}
\item[(i)] $\omega\in C( [0,+\infty); L^s(D))$ for all $s\in[1,+\infty)$;
\item[(ii)] $\omega(\cdot,t)\in \mathcal{R}_{\omega_0}$ for all $t\in[0,+\infty)$, where
\begin{equation}\label{rc}
 \mathcal{R}_{\omega_0}:=\{ w\in L^1_{{\rm loc}}(D)\mid  |\{w>a\}|=|\{\omega_0>a\}|\,\, \forall \,a\in \mathbb R\},
 \end{equation}
 is the rearrangement class of $\omega_0$. Here $\{w>a\}$ stands for $\{x\in D\mid w(x)>a\}$ and $|\cdot|$ denotes the two-dimensional Lebesgue measure;
\item[(iii)]  $ E(\omega(t,\cdot))=E(\omega_0)$ for all $t\in[0,+\infty)$, where
\[E(\omega(t,\cdot)):=\frac{1}{2}\int_D\omega(t,x) \mathcal G\omega (t,x)dx\]
is the kinetic energy of the fluid.

\end{itemize}
 \end{Yudovich's Theorem}
 Note that for any $s\in(1,+\infty)$ the functional $E$ is well-defined and weakly sequentially continuous in $L^s(D)$, in the sense that
 \[\lim_{n\to+\infty}E(w_n)= E(w)\]
 whenever $w_n$ converges  to $w$ weakly in $L^s(D)$ as $n\to+\infty$.
 This can be easily verified by applying $L^p$ estimate and Sobolev embedding theorem.

 According to Yudovich's theorem, any weak solution to the vorticity equation moves on some \emph{isovortical surface}, a set of functions with the same distribution function. As a consequence, the $L^s$ norm of $\omega(t,\cdot),$ where $ 1\leq s\leq +\infty,$ is independent of $t$. This conservative property plays a very important role in our stability analysis.

\subsection{Steady solution and stability}

If a weak solution is independent of the time variable, then it is called a steady solution.
It can be verified that $\omega\in L^\infty(D)$ is a steady solution if and only if
\begin{equation}\label{ss}
\int_D\omega\nabla^{\perp}\mathcal G\omega\cdot \nabla\phi dx=0,\quad\forall\,\phi\in C_c^\infty(D).
\end{equation}

Steady solutions of the vorticity equation  are very common. In fact, if $\omega\in L^\infty(D)$ satisfies
\begin{equation}\label{ss1}
\omega=g(\mathcal G\omega) \mbox{ a.e. in } D,
\end{equation}
where  either $g:\mathbb R\to\mathbb R$ is Lipschitz continuous   or $g:\mathbb R\to\mathbb R\cup\{\pm\infty\}$ is  monotone, then it   satisfies  \eqref{ss}. See Cao-Wang \cite{CW2} for a detailed proof. In terms of the stream function    $u:=\mathcal G\omega$,  \eqref{ss1} can be reformulated equivalently  as follows
\begin{equation}\label{ss2}
\begin{cases}
-\Delta u=g(u)& x\in D,\\
u=0&x\in\partial D.
\end{cases}
\end{equation}
For this reason, construction of steady Euler flows by  studying the semilinear elliptic problem \eqref{ss2} has been an important research topic in mathematical fluid mechanics. See \cite{BPW, CLW, CPY, CW1,LYY, SV} and the references therein.

Given a steady solution, an important related question is to study its nonlinear stability, also called stability of Lyapunov type. Informally,  a steady solution $\tilde\omega$ is said to be nonlinearly stable if for any initial vorticity sufficiently ``close" to $\tilde\omega$, the evolved vorticity remains ``close" to $\tilde\omega$ for all time. In many situations, it is usually more convenient and reasonable to consider a more general  stability concept, i.e., orbital stability for a   set of steady solutions. Below is the precise definition.

 \begin{definition}
Let $\mathcal M\subset L^\infty(D)$ be a set of steady solutions to the vorticity equation (i.e., each $\tilde\omega\in\mathcal M$ satisfies  \eqref{ss}), $\mathcal P$ be a subset of $L^\infty(D)$, and $\|\cdot\|$ be a norm in $L^\infty(D)$.
If for any $\varepsilon>0,$ there exists $\delta>0,$ such that for any $\omega_0\in \mathcal P$ satisfying
\[\inf_{w\in\mathcal M}\|\omega_0-w\|<\delta,\]
 it holds that
 \[\inf_{w\in\mathcal M}\|\omega(t,\cdot)-w\|<\varepsilon \,\,\mbox{ for all }t>0,\]
then $\mathcal M$ is said to be orbitally stable in the norm $\|\cdot\|$ with respect to initial perturbations in $\mathcal P$. Here $\omega(t,x)$ is the unique weak solution to the vorticity equation with initial vorticity $\omega_0.$

\end{definition}
When $\mathcal M$ contains only a steady solution $\tilde\omega$ and is orbitally stable, we say  $\tilde\omega$ is nonlinearly stable.

\begin{remark}
In many cases, $\mathcal M$ is invariant with respect to some continuous symmetry group. For example, orbital stability of a set of steady symmetric vortex pairs in the whole plane, which are invariant under translations parallel to some line,  can be found in \cite{Abe,Bu1,Bu2}; orbital stability of a set of concentrated vortex patches in a disk, which are invariant under rotations with respect to the origin of the disk,   has been investigated in \cite{CWWZ}.
\end{remark}

\begin{remark}
Obviously, the larger the perturbation class $\mathcal P$ is, the more stable $\mathcal M$ is. However, in general there is no direct relation between the stabilities in two different norms, even one is stronger than the other. An example can be found in Lin \cite{LZ2}, Section 4.
\end{remark}
In this paper, we mainly consider the stabilities in two norms, i.e., the $L^s$ norm of the vorticity $\|\omega\|_{L^s(D)}$ and energy norm $\|\omega\|_E$ defined as follows
\begin{equation}\label{en11}
\|\omega\|_E=\left(\int_D|\nabla\mathcal G\omega|^2dx\right)^{1/2}=(2E(\omega))^{1/2}.
\end{equation}

The systematic study of nonlinear stability of planar steady Euler flows dates back to Arnol'd \cite{A1,A2} in the 1960s.
To deal with the stability of steady Euler flows related to solutions of \eqref{ss2}, where $g\in C^1(\mathbb R)$ and is  strictly increasing, Arnol'd established what is now usually called the energy-Casimir functional method and proved several stability criteria. Arnol'd's idea was to study the energy-Casimir functional
\[EC(w)=E(w)-\int_DF(w)dx\]
related to \eqref{ss2}, where $F(s)=\int_0^sg^{-1}(\tau)d\tau.$  By using the fact that $EC$ is a flow invariant, Arnol'd proved that if some solution $\tilde u$ of \eqref{ss2} satisfies
\[\sup_{\mathbb R}g'(\tilde u)<\lambda_1,\]
where $\lambda_1$ is the first eigenvalue of $-\Delta$ in $D$ with zero Dirichlet boundary condition, then the corresponding flow is stable in the $L^2$ norm of the vorticity. This result is  called Arnol'd's second stability theorem. In 1990s, Wolansky and Ghil \cite{WG0,WG} introduced the method of supporting functionals and showed that to prove stability  in the $L^2$ norm of the vorticity it suffices to require the first eigenvalue of the corresponding linearized operator $-\Delta-g'(\tilde u)$ in $L^2(D)$ to be positive, that is, there exists some $c_0>0$ such that
\begin{equation}\label{ycst}
\int_D|\nabla \phi|^2 dx-\int_Dg'(\tilde u)\phi^2 dx\geq c_0\int_D\phi^2dx,\,\,\forall\,\phi\in H^1_0(D).
\end{equation}
Recently, Wang \cite{WA} improved Wolansky and Ghil's result further, showing that the first eigenvalue of  $-\Delta-g'(\tilde u)$ in $L^2(D)$ being nonnegative is sufficient to ensure stability; moreover, the stability obtained in \cite{WA} is in terms of the $L^s(D)$ norm of the vorticity for any $s\in(1,+\infty).$

The results in \cite{A1,A2,WA, WG0,WG} provide some general stability criteria  that cover a large class of planar steady Euler flows, but they still  have serious limitations. On the one hand, all of these criteria require $g$ to be a strictly increasing function, and thus can not be used to deal with steady Euler flows with compactly supported vorticity. On the other hand, the assumption that the first eigenvalue of  $-\Delta-g'(\tilde u)$ in $L^2(D)$ is nonnegative is too strong, and excludes many interesting cases.  A typical example is the
steady Euler flows associated with solutions of the Lane-Emden equation \eqref{le} below, which is exactly the topic of this paper.
We will see in Remark \ref{nsst} that for any nontrivial solution to the Lane-Emden equation the first eigenvalue of the corresponding linearized operator  must be negative, hence nonlinear stability can not obtained by  applying these criteria directly. To obtain stability, new idea  needs to be introduced.

\subsection{Lane-Emden equation and least energy solutions}
Consider the following Lane-Emden equation with zero Dirichlet boundary condition
\begin{equation}\label{le}
\begin{cases}
-\Delta u=|u|^{p-1}u &x\in D,\\
u=0&x\in\partial D,
\end{cases}
\end{equation}
where $1<p<+\infty.$ An $H^1_0$ weak solution of \eqref{le} is a function $u\in H^1_0(D)$ such that
\begin{equation}\label{wsle}
\int_D\nabla u\cdot\nabla \phi dx=\int_D|u|^{p-1}u\phi dx,\,\,\forall\,\phi\in H^1_0(D).
\end{equation}
Note that by Sobolev embedding $u\in H^1_0(D)$ implies $u\in L^s(D)$ for any $s\in[1,+\infty),$ which means $|u|^{p-1}u\in L^s(D)$ for any $s\in[1,+\infty),$ thus the integral in \eqref{wsle} makes sense. Also note that by elliptic regularity theory any $H^1_0$ weak solution at least belongs to $C^2(\bar D)$, thus is in fact a classical solution.

Our main purpose is to investigate the stability of steady Euler flows associated with solutions of \eqref{le}. However, without extra information, this is quite a tough task. 
In this paper, we will only be focusing on a special class of solutions with certain ``minimality" condition, i.e., least energy solutions of \eqref{le}. Recall that a solution $u$ of \eqref{le} is called a \emph{least energy solution}  if
\begin{equation}\label{ccb}
\mathcal I(u)=\inf\{\mathcal I(\phi)\mid \phi\in \mathcal N\},
\end{equation}
where $ \mathcal I:H^1_0(D)\to\mathbb R$ is the natural energy functional related to \eqref{le}
\[\mathcal I(\phi)=\frac{1}{2}\int_D|\nabla\phi|^2dx-\frac{1}{p+1}\int_D|\phi|^{p+1}dx,\quad \phi\in H^1_0(D),\]
and $\mathcal N$ is the Nehari manifold defined by
\[\mathcal N=\left\{\phi\in H^1_0(D)\mid\phi\neq 0,\,\, \int_D|\nabla\phi|^2 dx=\int_D|\phi|^{p+1}dx\right\}.\]

Note that  for any $p\in(1,\infty)$, the Nehari manifold $\mathcal N$ is not empty, and $\mathcal I$  attains its minimum value over $\mathcal N$. Moreover, if $u$ is a minimizer of $\mathcal I$ over $\mathcal N$, then $ u$ must be of constant sign in $D$ (i.e., either $ u>0$ in $D$ or $u<0$ in $D$) and solve the Lane-Emden equation \eqref{le}. These results are classical in the theory of nonlinear elliptic equations, and we refer the interested reader to \S 2.3.2 in \cite{BS} for detailed proofs of these facts.
From these results, we immediately deduce that for a solution $ u$ of \eqref{le}, \eqref{ccb} holds if and only if
\[\mathcal I(u)=\inf\{\mathcal I(\phi)\mid \phi\in H^1_0(D)\setminus\{0\} \mbox{ is a solution of \eqref{le}}\}.\]
This is the reason why $u$ is called a least energy solution.

\begin{remark}\label{nsst}
For any nontrivial solution to $ u$ of \eqref{le}, the first eigenvalue of the corresponding linearized operator $-\Delta-p|u|^{p-1}$  must be negative. In fact, by choosing $\phi=u$ in \eqref{ycst} we get
\[\int_D|\nabla u|^2 dx-p\int_D|u|^{p+1} dx=(1-p)\int_D|u|^{p+1} dx<0.\]
\end{remark}

Denote $\mathcal L_p$ the set of all least energy solutions of \eqref{le}, or equivalently, $\mathcal L_p$ is the set of all minimizers of $\mathcal I$ over $\mathcal N$.
Our aim in this paper is to study the orbital stability of the set of flows whose stream functions belong to $\mathcal L_p$. Since we will mainly work in the vorticity space, it is convenient to define
\[\mathcal  W_p=\{-\Delta u\mid u\in\mathcal L_p\}.\]

Denote
\begin{equation}\label{ccd}
c_p=\inf\{\mathcal I(\phi)\mid \phi\in \mathcal N\},
\end{equation}
which is a positive number depending only on $p$.
For any least energy solution $u$, it holds that
\[\int_D|\nabla u|^2 dx=\int_D|u|^{p+1}dx,\quad \frac{1}{2}\int_D|\nabla u|^2dx-\frac{1}{p+1}\int_D|u|^{p+1}dx=c_p,\]
from which  we get by a simple computation
\[\|u\|_{L^{p+1}}=\left(\frac{2c_p(p+1)}{p-1}\right)^{\frac{1}{p+1}},\]
and
\begin{equation}\label{cce}
\|-\Delta u\|_{L^{1+\frac{1}{p}}(D)}=\||u|^{p-1}u\|_{L^{1+\frac{1}{p}}}=\|u\|_{L^{p+1}}^p=\left(\frac{2c_p(p+1)}{p-1}\right)^{\frac{p}{p+1}}:=\mu_p.
\end{equation}
In the rest of this paper, we denote
\begin{equation}\label{dofsp}
\mathcal S_p=\left\{w\in L^{1+\frac{1}{p}}(D)\mid \|w\|_{ L^{1+\frac{1}{p}}(D)}\leq \mu_p\right\}.
\end{equation}
Obviously $\mathcal W_p\subset \mathcal S_p$.

\subsection{Main results}

Having made enough preparations, we are ready to state our main results. Our first result is about the orbital stability of $\mathcal W_p$ in the $L^s$ norm of the vorticity.

\begin{theorem}\label{thm1}
Let $s\in(1,+\infty)$. Then $\mathcal W_p$ is orbitally stable in the $L^s$ norm of the vorticity with respect to initial perturbations in $L^\infty(D).$
\end{theorem}

\begin{remark}
The assumption $1<s<+\infty$ is necessary in our proof, since we frequently use the property that $L^s(D)$ is reflexive. However, if we choose the perturbation class to be some bounded subset of $L^\infty(D)$, then $s$ can be any number in $[1,+\infty)$, since for any $r,s\in[1,+\infty)$, the $L^{r}$ norm and the $L^{s}$ norm are equivalent norms on any fixed bounded subset of $L^\infty(D)$.
\end{remark}

Our second result is about the orbital stability of $\mathcal W_p$ in the energy norm. Unlike Theorem \ref{thm1}, in this case we can only prove orbital stability with respect to a smaller perturbation class.

\begin{theorem}\label{thm2}
 $\mathcal W_p$ is orbitally stable in the energy norm with respect to initial perturbations in $\mathcal S_p\cap L^\infty(D)$.
\end{theorem}

The proofs of Theorems \ref{thm1} and \ref{thm2} include two key ingredients. The first one is a new variational characterization of the set of least energy solutions $\mathcal W_p$ in terms of the vorticity. More precisely, we will show that  $\mathcal W_p$ is exactly the set of maximizers of $E$ over $\mathcal S_p$. The second one is compactness. We will prove that for any maximizing sequence of $E$ over $\mathcal S_p$, there  exist a subsequence that converges strongly to some maximizer.
Having these two key ingredient,   we can   prove Theorems \ref{thm1} and \ref{thm2} by properly using flows invariants of the Euler equation. It is worth noting that suitable variational characterization and compactness argument are  also the key  ingredients in the
 stability analysis for  other planar steady Euler flows. For example, see \cite{BG} for
flows related to solutions of the mean field  equation,  \cite{CW1,CWN,WA2} for   flows with sharply concentrated vorticity, and a recent work \cite{Abe} concerning Lamb's circular vortex pairs.

Based on Theorems \ref{thm1} and \ref{thm2}, we can prove the nonlinear stability of steady Euler flows related to isolated least energy solutions.
For any $\tilde\omega\in \mathcal W_p$, it is said to be isolated if there is a positive distance between $\tilde\omega$ and $\mathcal W_p\setminus \{\tilde\omega\}$, i.e.,
\begin{equation}\label{isl}
\inf_{w\in\mathcal W_p\setminus \{\tilde\omega\}}\|w-\tilde\omega\|_{L^1(D)}>0.
\end{equation}
We will see in Lemma \ref{eqnms} that  in \eqref{isl}  the $L^1$ norm can be replaced by the $L^s$ norm for any  $s\in[1,+\infty)$ or the energy norm $\|\cdot\|_E.$

\begin{theorem}\label{thm3}
Let $\tilde\omega\in \mathcal W_p$ be isolated in the sense of \eqref{isl}.
 Then $\tilde\omega$ is stable in the $L^s$ norm of the vorticity with respect to initial perturbations in $L^\infty(D)$ for any fixed $s\in(1,+\infty)$.
\end{theorem}

\begin{theorem}\label{thm4}
Let $\tilde\omega\in \mathcal W_p$ be isolated in the sense of \eqref{isl}.
 Then $\tilde\omega$ is  stable in the energy norm  with respect to initial perturbations in $\mathcal S_p\cap L^\infty(D)$.
\end{theorem}

By Theorems \ref{thm3} and \ref{thm4}, to prove the stability of a planar steady Euler flow relate to some given least energy solution of \eqref{le}, it suffices to show that this least energy solution is isolated in the sense of \eqref{isl}. However, as far as we know,  there is no general result on the isolatedness of least energy solutions for general $D$ and $p$. In fact, when $D$ is an annulus, as $p\to+\infty$ least energy solutions of \eqref{le} will concentrate at some global minimum point of the Robin function of $D$ (see \cite{RW1, RW2} for example), hence any least energy solution can not be isolated due to rotational invariance (in an annulus the motion of a real fluid can not be described by the vorticity equation \eqref{ve} anymore, but it is still interesting to discuss this example mathematically).
However, it has been proved in the literature that  for some special domains or exponents the positive least energy solution of \eqref{le} is unique, and thus must be isolated.

\begin{theorem}[\cite{DA}\cite{LCS}]\label{thm5}
Consider the following semilinear elliptic problem
\begin{equation}\label{ple}
\begin{cases}
-\Delta u=u^p&x\in D,\\
u>0&x\in D,\\
u=0&x\in\partial D.
\end{cases}
\end{equation}
Then
\begin{itemize}
\item[(i)] for any smooth bounded convex domain $D$ and $p\in(1,+\infty),$  \eqref{ple} admits a unique least energy solution;
\item[(ii)] for any smooth bounded domain $D$, there exists some $\delta_0>0$, such that for any $p\in(1,1+\delta_0)$, \eqref{ple} admits a unique solution.
\end{itemize}
 \end{theorem}
The proof of (i) can be found in Lin \cite{LCS}, and the proof of (ii) can be found in Dancer \cite{DA}.

\begin{remark}
It has been proved in \cite{DGIP} that for any smooth bounded convex domain $D$, there exists some $p_0>0$, such that for any $p>p_0$, \eqref{ple}  has a unique solution. However, for general $p$ the uniqueness problem related to \eqref{ple} is still open.
\end{remark}

A direct consequence of Theorems \ref{thm3}, \ref{thm4} and \ref{thm5} is the following theorem.
\begin{theorem}\label{coro6}
Suppose $D$ and $p$ satisfy the assumption of (i) or (ii) in Theorem \ref{thm5}. Let $\tilde u$ be the unique solution to \eqref{ple} and denote $\tilde\omega=-\Delta\tilde u$. Then
\[\mathcal W_p=\{\tilde\omega, -\tilde\omega\}.\]
 As a result,
 $\tilde\omega$ is stable in the $L^s$ norm of the vorticity with respect to initial perturbations in $L^\infty(D)$ for any fixed $s\in(1,+\infty)$, and is also  stable in the energy norm $\|\cdot\|_E$ with respect to initial perturbations in $\mathcal S_p\cap L^\infty(D)$.

\end{theorem}

This paper is organized as follows. In Section 2, we give a new variational characterization of $\mathcal W_p$, which is essential for later proofs . In Section 3, we prove Theorems \ref{thm1} and \ref{thm2}. In Section 4, we prove Theorems \ref{thm3} and \ref{thm4}. In Section 5, we discuss the analogous case of $p\in(0,1]$.

\section{New variational characterization of least energy solutions}

In this section, we give least energy solutions a new equivalent variational characterization in terms of suitable flow invariants. To be precise, our task in this section is to prove the following proposition

\begin{proposition}\label{prop1}
It holds that
$\mathcal W_p=\{w\in\mathcal S_p\mid E(w)=M_p\},$ where
\begin{equation}\label{maxf}
M_p=\sup\{E(w)\mid w\in \mathcal S_p\}.
\end{equation}
\end{proposition}

To make it clear, we divide the proof into several lemmas.

\begin{lemma}\label{le222}
$E$ attains its maximum value over $\mathcal S_p$, and any maximizer $\tilde\omega$ satisfies $\|\tilde\omega\|_{L^{1+\frac{1}{p}}(D)}=\mu_p.$
\end{lemma}

\begin{proof}
First by  H\"older's inequality, $L^p$ estimate and Sobolev embedding theorem, it holds for any
$w\in \mathcal S_p$ that
\[|E(w)|\leq \frac{1}{2}\|w\|_{L^{1+\frac{1}{p}}(D)}\|\mathcal Gw\|_{L^{p+1}(D)}
\leq C_1\|w\|_{L^{1+\frac{1}{p}}(D)}\|\mathcal Gw\|_{L^\infty(D)}\leq C_2\|w\|^2_{L^{1+\frac{1}{p}}(D)}\leq C_2\mu_p^2,\]
where $C_1, C_2$ are positive constants depending only on $p$ and $D$. Therefore
\[\sup_{w\in\mathcal S_p}E(w)<+\infty.\]
Now let $\{w^n\}_{n=1}^{+\infty}\subset\mathcal S_p$ be a sequence such that
\[\lim_{n\to+\infty}E(w^n)=\sup_{w\in\mathcal S_p}E(w).\]
Since $\mathcal S_p$ is obviously a bounded and weakly closed subset of $L^{1+\frac{1}{p}}(D)$, we can choose a subsequence $\{w^{n_j}\}_{j=1}^{+\infty}$ such that
$w^{n_j}$ converges to some $\tilde\omega\in \mathcal S_p$ weakly  in $L^{1+\frac{1}{p}}(D)$. Taking into account the fact that $E$ is weakly sequentially continuous in $L^{1+\frac{1}{p}}(D)$, we get
\[\sup_{w\in\mathcal S_p}E(w)=\lim_{j\to+\infty}E(w^{n_j})=E(\tilde\omega),\]
which means $E$ attains its maximum value at $\tilde\omega$.

Now we verify that any maximizer $\tilde\omega$ satisfies $\|\tilde\omega\|_{L^{1+\frac{1}{p}}(D)}=\mu_p$. Suppose otherwise that
\[0<\|\tilde\omega\|_{L^{1+\frac{1}{p}}(D)}<\mu_p,\]
 then it is easy to check that
$$\frac{\mu_p\tilde\omega}{\|\tilde\omega\|_{L^{1+\frac{1}{p}}(D)}}\in\mathcal S_p$$
and
\[E\left(\frac{\mu_p\tilde\omega}{\|\tilde\omega\|_{L^{1+\frac{1}{p}}(D)}}\right)=\frac{\mu_p^2}{\|\tilde\omega\|^2_{L^{1+\frac{1}{p}}(D)}}
E(\tilde\omega)>E(\tilde\omega),\]
which is a contradiction.

\end{proof}

\begin{lemma}\label{profi}
For any maximizer $\tilde\omega$ of $E$ over $\mathcal S_p$, either $\tilde \omega>0$ a.e. in $D$ or $\tilde \omega<0$ a.e. in $D$. If $\tilde \omega>0$ a.e. in $D$, then
\[\tilde\omega=(2M_p)^{-{p}}\mu_p^{p+1}(\mathcal G\tilde\omega)^p  \quad\mbox{ a.e. in }  D.\] If $\tilde \omega<0$ a.e. in $D$, then
\[\tilde\omega=-(2M_p)^{-{p}}\mu_p^{p+1}(-\mathcal G\tilde\omega)^p\quad \mbox{ a.e. in } D.\]

\end{lemma}
\begin{proof}
First we show that for any maximizer $\tilde\omega$, it holds that
\begin{equation}\label{epen}
\tilde\omega\geq 0\,\,\mbox{ a.e. in}\,\, D\,\,\mbox{ or }\,\,\tilde\omega\leq 0 \,\,\mbox{ a.e. in}\,\, D.
\end{equation}
Denote $\tilde\omega^+=\max\{\tilde\omega,0\}, \tilde\omega^-=-\min\{\tilde\omega,0\}$. Obviously $\tilde\omega=\tilde\omega^+-\tilde\omega^-$ and $|\tilde\omega|=\tilde\omega^++\tilde\omega^-.$ We compare $E(\tilde\omega)$ and $E(|\tilde\omega|)$ as follows
\begin{equation}\label{242}
E(|\tilde\omega|)-E(\tilde\omega)=\int_D\tilde\omega^+\mathcal G\tilde\omega^-+\tilde\omega^-\mathcal G \tilde\omega^+dx\geq 0.
\end{equation}
By the strong maximum principle, the inequality in \eqref{242} is an equality if and only $\tilde\omega^+= 0$ a.e. in $D$ or $\tilde\omega^-= 0$ a.e. in $D$, which proves \eqref{epen}.

To continue, we first assume that $\tilde\omega\geq 0$ a.e. in $D$.
Below we show that in this case
\begin{equation}
\tilde \omega>0 \quad\mbox{ a.e. in } D,
 \end{equation}
\begin{equation}
\tilde\omega=(2M_p)^{-p}\mu_p^{p+1}(\mathcal G\tilde\omega)^p\quad \mbox{ a.e. in }D.
\end{equation}
  To this end, for any  $\phi\in C_c^\infty(D)$, define a family of test functions
\[\omega^\varepsilon:=\frac{\tilde\omega+\varepsilon\phi}{\|\tilde\omega+\varepsilon\phi\|_{L^{1+\frac{1}{p}}(D)}},\]
where $\varepsilon\in\mathbb R$ is small in absolute value such that $\|\tilde \omega+\varepsilon\phi\|_{L^{1+\frac{1}{p}}(D)}>0$.
 Since $\tilde\omega$ is a maximizer and $\omega^\varepsilon|_{\varepsilon=0}=\tilde\omega$, we have
\[\frac{d}{d\varepsilon}E(\omega^\varepsilon)\bigg|_{\varepsilon=0}=0.\]
On the other hand, we can calculate $\frac{d}{d\varepsilon}E(\omega^\varepsilon)\big|_{\varepsilon=0}$ as follows
\begin{equation*}
\begin{split}
\frac{d}{d\varepsilon}E(\omega^\varepsilon)\bigg|_{\varepsilon=0}
&=\frac{1}{2}\frac{d}{d\varepsilon}\int_D\frac{\tilde\omega+\varepsilon\phi}{\|\tilde\omega+\varepsilon\phi\|_{L^{1+\frac{1}{p}}(D)}}\frac{\mathcal G(\tilde\omega+\varepsilon\phi)}{\|\tilde\omega+\varepsilon\phi\|_{L^{1+\frac{1}{p}}(D)}}dx\bigg|_{\varepsilon=0}\\
&=-\mu_p^{-3-\frac{1}{p}}\int_D\tilde\omega\mathcal G\tilde\omega dx\int_D\tilde\omega^\frac{1}{p}\phi dx+\mu^{-2}_p\int_D\phi\mathcal G\tilde\omega dx\\
&=-2M_p\mu_p^{-3-\frac{1}{p}}\int_D\tilde\omega^\frac{1}{p}\phi dx+\mu^{-2}_p\int_D\phi\mathcal G\tilde\omega dx.
\end{split}
\end{equation*}
Therefore we obtain
\[\int_D\tilde\omega^{\frac{1}{p}}\phi dx=(2M_p)^{-1}\mu_p^{1+\frac{1}{p}}\int_D\phi\mathcal G\tilde\omega dx,\,\,\forall\,\phi\in C_c^\infty(D),\]
which implies
\begin{equation}\label{bom}
\tilde\omega=(2M_p)^{-p}\mu_p^{p+1}\mathcal (G\tilde\omega)^p\,\,\mbox{ a.e. in } D.
\end{equation}
Finally, by the strong maximum principle, we have $\tilde \omega>0$ a.e. in $D$.

For the case of $\tilde \omega\leq 0$ a.e. in $D$, the proof is almost identical as above.
\end{proof}

\begin{lemma}\label{lem24}
$\mu_p^{1+\frac{1}{p}}=2M_p.$
\end{lemma}
\begin{proof}
First we show that
 \begin{equation}\label{ellig}
 2M_p\leq \mu_p^{1+\frac{1}{p}}.
 \end{equation}
 Let $\tilde \omega$ be a positive maximizer of $E$ over $\mathcal S_p.$ By Lemma \ref{profi}, $\tilde u:=\mathcal G\tilde \omega$ satisfies
\[-\Delta \tilde u=\lambda_p \tilde u^p,\,\,\lambda_p=(2M_p)^{-p}\mu_p^{p+1}.\]
Obviously $\bar u:=\lambda_p^{\frac{1}{p-1}}\tilde u$ is a solution of \eqref{le}. By the definition of least energy solutions, we have
 \begin{equation}
 \mathcal I(\bar u)\geq c_p,
  \end{equation}
where $\mathcal I$ and $c_p$ are defined by \eqref{ccb} and \eqref{ccd} in Section 1. A direct computation gives
 $$(\frac{1}{2}-\frac{1}{p+1})(2M_p)^{\frac{1+p}{1-p}}\mu_p^{\frac{2(p+1)}{p-1}}\geq c_p.$$
Taking into account the relation \eqref{cce}, we obtain \eqref{ellig}.

Now we prove the inverse inequality
 \begin{equation}\label{ellig2}
2M_p\geq \mu_p^{1+\frac{1}{p}}.
 \end{equation}
  Choose $\tilde\omega\in \mathcal W_p$ such that $\tilde\omega>0$ in $D$.  Obviously $\tilde \omega$ satisfies
\begin{equation}\label{cc11}
\tilde \omega=(\mathcal G\tilde\omega)^p\,\, \mbox{ a.e. in }D,
\end{equation}
\begin{equation}\label{cc12}
\|\tilde\omega\|_{L^{1+\frac{1}{p}}(D)}=\mu_p.
\end{equation}
Then
\begin{equation}\label{cc14}
\int_D\tilde\omega\mathcal G\tilde\omega dx=\int_D\tilde\omega^{1+\frac{1}{p}}=\mu_p^{1+\frac{1}{p}}.
\end{equation}
On the other hand, by the definition of $M_p$, we have
\begin{equation}\label{cc13}
\int_D\tilde\omega\mathcal G\tilde\omega dx\leq 2M_p.
\end{equation}
The desired inverse inequality \eqref{ellig2} follows from \eqref{cc13} and \eqref{cc14}.

\end{proof}

\begin{lemma}\label{lem25}
Let $\tilde \omega \in L^{1+\frac{1}{p}}(D)$ satisfying $\tilde\omega\geq 0$ a.e. in $D$. Then $\tilde\omega\in\mathcal W_p$ if and only if
$\tilde\omega=(\mathcal G\tilde\omega)^p$ a.e. in $D$ and $\|\tilde \omega\|_{L^{1+\frac{1}{p}}(D)}=\mu_p.$

\end{lemma}
\begin{proof}
The ``only if" part is obvious. Below we prove the ``if" part. Let $\tilde \omega \in L^{1+\frac{1}{p}}(D)$ satisfying
\[
\tilde\omega\geq 0\,\, \mbox{ a.e. in } D,\quad  \tilde\omega=(\mathcal G\tilde\omega)^p\,\,\mbox{ a.e. in }D,\quad  \|\tilde \omega\|_{L^{1+\frac{1}{p}}(D)}=\mu_p.
\]
 Denote $\tilde u=\mathcal G\tilde\omega$. We compute $\mathcal I(\tilde u)$ as follows
\begin{align*}
\mathcal I(\tilde u)&=\frac{1}{2}\int_D|\nabla \tilde u|^2dx -\frac{1}{p+1}\int_D|\tilde u|^{p+1}dx\\
&=\frac{1}{2}\int_D\tilde\omega\mathcal G\tilde\omega dx-\frac{1}{p+1}\int_D\tilde\omega^{1+\frac{1}{p}}dx\\
&=\left(\frac{1}{2}-\frac{1}{p+1}\right)\int_D\tilde\omega^{1+\frac{1}{p}}dx\\
&=\left(\frac{1}{2}-\frac{1}{p+1}\right)\mu_p^{1+\frac{1}{p}}.
\end{align*}
Note that in the second equality we have used  integration by parts. Taking into account the relation (see \eqref{cce})
\[\mu_p=\left(\frac{2c_p(p+1)}{p-1}\right)^{\frac{p}{p+1}},\]
we get $\mathcal I(\tilde u)=c_p,$ which means that $\tilde u$ is a least energy solution of \eqref{le}, and thus $\tilde \omega\in\mathcal W_p.$
\end{proof}

\begin{proof}[Proof of Proposition \ref{prop1}] It  is an immediate consequence of Lemmas \ref{profi}-\ref{lem25}.
\end{proof}

\section{Proofs of Theorems \ref{thm1} and \ref{thm2}}
In this section we prove  Theorems \ref{thm1} and \ref{thm2}. The proofs are mostly based on the variational characterization of $\mathcal W_p$ proved in Proposition \ref{prop1} and conservative properties of the vorticity equation stated in Yudovich's theorem.

To make the proof clear, we give several lemmas first.
\begin{lemma}\label{301}
$\mathcal W_p$ is a bounded subset of $L^\infty(D).$
\end{lemma}
\begin{proof}
First recall that $\mathcal W_p$ is bounded in $L^{1+\frac{1}{p}}(D)$. By $L^p$ estimate and Sobolev embedding theorem  $\mathcal G$ is a bounded linear operator from $L^{1+\frac{1}{p}}(D)$ to $L^\infty(D)$, thus for any $w\in\mathcal W_p$ it holds that
\[\|\mathcal Gw\|_{L^\infty(D)}\leq C\]
for some $C>0$ depending only on $p$ and $D$. Taking into account the fact that any $w\in\mathcal W_p$ satisfies \eqref{le}, we obtain
\[\|w\|_{L^\infty(D)}=\||\mathcal Gw|^{p-1}\mathcal G\omega\|_{L^\infty(D)}\leq C^p.\]
This means that $\mathcal W_p$ is bounded in $L^\infty(D).$
\end{proof}

\begin{lemma}\label{ufcu}
Let $s\in(1,+\infty]$ be fixed. Then there exists a positive number C, depending only on $s$ and $D$, such that
\[\|w\|_E\leq C\|w\|_{L^s(D)},\,\,\forall\,w\in L^s(D).\]
\end{lemma}
\begin{proof}
Without loss of generality, assume that $s\in(1,2).$  Then
 \[\|w\|_E=\|\nabla\mathcal Gw\|_{L^2(D)}\leq C\|\nabla\mathcal Gw\|_{W^{1,s}(D)}\leq C\|\mathcal Gw\|_{W^{2,s}(D)}\leq C\|w\|_{L^s(D)},\]
 where $C>0$ depends only  on $s$ and $D$. Note that  the first inequality follows from the Sobolev embedding $W^{1,s}(D)\hookrightarrow L^2(D)$, and the third inequality follows from  standard elliptic regularity theory.
  \end{proof}


The following lemma provides the necessary compactness   for the proofs of Theorems \ref{thm1} and \ref{thm2}.

\begin{lemma}\label{scmcs}
Let $\{w_n\}_{n=1}^{+\infty}\subset \mathcal S_p$ be a sequence satisfying
\[\lim_{n\to+\infty}E(w_n)=M_p.\]
Then there exist a subsequence $\{w_{n_j}\}_{j=1}^{+\infty}$ and some $\eta\in \mathcal W_p$ such that $w_{n_j}$ converges strongly to $\eta$ in $L^{1+\frac{1}{p}}(D)$ as $j\to+\infty$.
\end{lemma}
\begin{proof}
Since $\mathcal S_p$ is a weakly sequentially compact subset of $L^{1+\frac{1}{p}}(D)$, there exist a subsequence $\{w_{n_j}\}_{j=1}^{+\infty}$ and some $\eta\in \mathcal S_p$ such that  as $j\to+\infty$
\begin{equation}\label{wc55}
w_{n_j}\rightharpoonup\eta\,\, \mbox { in } L^{1+\frac{1}{p}}(D).
\end{equation}
From \eqref{wc55}, we get
\begin{equation}\label{x100}
\|\eta\|_{L^{1+\frac{1}{p}}(D)}\leq \liminf_{j\to+\infty}\|w_{n_j}\|_{L^{1+\frac{1}{p}}(D)}.
\end{equation}
and
\begin{equation}\label{x101}
E(\eta)=\lim_{n\to+\infty}E(w_n)=M_p,
\end{equation}
By \eqref{x101} we get
$\eta\in \mathcal W_p$ and
\begin{equation}\label{x102}
\|\eta\|_{L^{1+\frac{1}{p}}(D)}=\mu_p.
\end{equation}
Now \eqref{x100} and \eqref{x102} together give
\[\mu_p=\|\eta\|_{L^{1+\frac{1}{p}}(D)}\leq \liminf_{j\to+\infty}\|w_{n_j}\|_{L^{1+\frac{1}{p}}(D)}\leq \limsup_{j\to+\infty}\|w_{n_j}\|_{L^{1+\frac{1}{p}}(D)}\leq \mu_p,\]
which yields
\begin{equation}\label{x105}
\lim_{j\to+\infty}\|w_{n_j}\|_{L^{1+\frac{1}{p}}(D)}=\|\eta\|_{L^{1+\frac{1}{p}}(D)}=\mu_p.
\end{equation}
From \eqref{wc55} and \eqref{x105} we immediately obtain strong convergence, which completes the proof.
\end{proof}
From Lemma \ref{scmcs}, we can easily prove the following
\begin{lemma}\label{bcyx}

For any $s\in[1,+\infty)$, $\mathcal W_p$ is compact in $L^{s}(D)$.

\end{lemma}
\begin{proof}
By Lemma \ref{scmcs}, it is clear that $\mathcal W_p$ is compact in $L^{1+\frac{1}{p}}(D)$. Taking into account Lemma \ref{301}, we deduce that  $\mathcal W_p$ is compact in $L^{s}(D)$ for any $s\in[1,+\infty)$.

\end{proof}


In Section 2, we only show that $\mathcal W_p$ is the set of maximizers of $E$ over $\mathcal S_p$. However, to prove Theorem \ref{thm1} we need to consider initial perturbations in $L^\infty(D)$. To overcome this difficulty, we used the method of ``followers" introduced by Burton in \cite{B5}.
To begin with, we need the following existence and uniqueness result for linear transport equations proved by  Burton in \cite{B5}.
\begin{lemma}\label{lteb5}
Let $\omega\in L^\infty((0,+\infty)\times D)$ and $\zeta_0\in L^\infty(D)$. Then exists a unique  $\zeta\in L^\infty((0,+\infty)\times D)$ such that
\begin{itemize}
\item[(a)] $\zeta$ satisfies $\partial_t\zeta +\nabla^\perp \mathcal G\omega\cdot\nabla\zeta=0$ in the sense of distributions, that is,
\begin{equation}\label{ydmf}
\int_0^\infty\int_{D} \zeta\partial_t\phi+\zeta\nabla^\perp\mathcal G\omega\cdot\nabla\phi dxdt=0\quad \mbox { for any }\phi\in C_c^\infty((0,+\infty)\times D);
\end{equation}
\item [(b)] $\zeta\in C([0,+\infty); L^s(D))$ for any $s\in[1,+\infty)$;
\item[(c)] $\zeta(0,\cdot)=\zeta_0;$
\item[(d)] $\zeta(t,\cdot)\in\mathcal R_{\zeta_0}$ for any $t\in[0,+\infty).$
\end{itemize}
\end{lemma}
\begin{proof}
See Lemmas 11 and  12 in \cite{B5}.
\end{proof}
Now we are ready to give the proof of Theorem \ref{thm1}.
\begin{proof}[Proof of Theorem \ref{thm1}]

Suppose by contradiction that $\mathcal W_p$ is not orbitally stable in $L^s$ norm with initial perturbations in $L^\infty(D)$. Then there exist a positive number $\varepsilon_0,$ a sequence of initial values $\{\omega_0^n\}_{n=1}^{+\infty}\subset L^\infty(D)$ and a sequence of positive numbers $\{t_n\}_{n=1}^{+\infty}$ such that
\begin{equation}\label{snt}
\inf_{w\in\mathcal W_p}\|\omega^n_0-w\|_{L^s(D)}<\frac{1}{n} \mbox{ for all }n
\end{equation}
and
 \begin{equation}\label{pn1}
\inf_{w\in\mathcal W_p}\|\omega^n(t_n,\cdot)-w\|_{L^s(D)}\geq\varepsilon_0 \mbox{ for all }n,
\end{equation}
where $\omega^n(t,x)$ is the unique solution to the vorticity equation with initial vorticity $\omega^n_0.$ By the conservation of vorticity (see (ii) in Yudovich's theorem in Section 1), it holds that
\begin{equation}\label{abcdd}
\omega^n(t_n,\cdot) \in\mathcal R_{\omega^n_0}\,\, \mbox{ for each $n$.}
\end{equation}

By \eqref{snt}, we can choose $\zeta^n_0\in\mathcal W_p$ such that
\begin{equation}\label{znii}
\|\zeta^n_0-\omega^n_0\|_{L^s(D)}<\frac{1}{n}\,\, \mbox{ for each $n$.}
\end{equation}
By Lemma \ref{301} it is clear that
\begin{equation}\label{em1123}
\{\zeta^n_0\}_{n=1}^{+\infty}\,\mbox{ is bounded  in $L^\infty(D)$}.
\end{equation}
Besides, by \eqref{znii} it is easy to verify that
\begin{equation}\label{em11}
\lim_{n\to+\infty}E(\omega^n_0)=\lim_{n\to+\infty}E(\zeta^n_0)=M_p.
\end{equation}
Using the conservation of energy (see (iii) in Yudovich's theorem), we get  from \eqref{em11} that
\begin{equation}\label{em12}
\lim_{n\to+\infty}E(\omega^n(t_n,\cdot))=M_p.
\end{equation}

By Lemma \ref{lteb5}, for each $n$ there exists $\zeta^n(t,x)\in L^\infty((0,+\infty)\times D)$ that solves \[\partial_t\zeta^n+\nabla^\perp\mathcal G\omega^n\cdot\nabla\zeta^n=0\] in the sense of distributions with initial data $\zeta^n(0,\cdot)=\zeta^n_0.$
Moreover, by (d) in Lemma \ref{lteb5} we have
\begin{equation}\label{ltf}
\zeta^n(t_n,\cdot) \in\mathcal R_{\zeta^n_0}\subset \mathcal S_p\,\, \mbox{ for all $n$}.
\end{equation}
It is also easy to see that
\begin{equation}\label{ydls}
\partial_t(\zeta^n-\omega^n)+\nabla^\perp\mathcal G\omega^n\cdot\nabla(\zeta^n-\omega^n)=0
\end{equation}
holds in the sense of distributions,
therefore again by (d) in Lemma \ref{lteb5} we obtain
\begin{equation}\label{lte2}
   \zeta^n(t_n,\cdot)-\omega^n(t_n,\cdot)\in\mathcal R_{\zeta^n_0-\omega^n_0} \,\,\mbox{ for all $n$}.
\end{equation}
To summarize, by using Lemma \ref{lteb5} we have constructed a sequence of ``followers" $\{\zeta^n\}_{n=1}^{+\infty}$  satisfying \eqref{ltf} and \eqref{lte2}.

By \eqref{em1123} and \eqref{ltf}, we see that
\begin{equation}\label{mzxf}
\{\zeta^n(t_n,\cdot)\}_{n=1}^{+\infty} \,\mbox{ is bounded in $L^\infty(D)$.}
\end{equation}
Besides, as a consequence of \eqref{znii} and \eqref{lte2}, it holds that
 \begin{equation}\label{ltg}
\|\zeta^n(t_n,\cdot)-\omega^n(t_n,\cdot)\|_{L^s(D)}=\|\zeta^n_0-\omega^n_0\|_{L^s(D)}<\frac{1}{n} \mbox{ for all }n.
\end{equation}

In view of \eqref{em12} and \eqref{ltg}, we obtain
\begin{equation}\label{em13}
\lim_{n\to+\infty}E(\zeta^n(t_n,\cdot))=M_p.
\end{equation}
Applying Lemma \ref{scmcs}, we deduce from \eqref{em13}  that there exist a subsequence $\{\zeta^{n_j}(t_{n_j},\cdot)\}_{j=1}^{+\infty}$ and some $\eta\in\mathcal W_p$ such that \begin{equation}\label{amfin}
\zeta^{n_j}(t_{n_j},\cdot)\to \eta \,\,\mbox{ in $L^{1+\frac{1}{p}}(D)$ as $j\to +\infty$},
\end{equation}
which together with \eqref{mzxf} gives
 \begin{equation}\label{mzxf2}
\zeta^{n_j}(t_{n_j},\cdot)\to \eta \,\,\mbox{ in $L^{s}(D)$ as $j\to +\infty$}.
\end{equation}

Now \eqref{ltg} and \eqref{mzxf2} yield
\[\lim_{j\to+\infty}\omega^{n_j}(t_{n_j},\cdot)=\eta \in\mathcal W_p,\]
which is a contradiction to \eqref{pn1}. Thus the proof is completed.
\end{proof}

\begin{remark}
If we only consider the smaller perturbation class $\mathcal S_p,$ then there is no  need to introduce the ``follower" $\zeta^n.$

\end{remark}

The proof of Theorem \ref{thm2} is based on
Theorem \ref{thm1}.

 \begin{proof}[Proof of Theorem \ref{thm2}]
Suppose that the conclusion is not true. Then there exist a positive number $\varepsilon_0,$ a sequence of initial values $\{\omega_0^n\}_{n=1}^{+\infty}\subset \mathcal S_p\cap L^\infty(D)$ and a sequence of positive numbers $\{t_n\}_{n=1}^{+\infty}$ such that
\begin{equation}\label{e00}
\inf_{w\in \mathcal W_p}\|w-\omega^n_0\|_{E}<\frac{1}{n}
\end{equation}
and
\begin{equation}\label{e01}
\inf_{w\in \mathcal W_p}\|w-\omega^n(t_n,\cdot)\|_E\geq\varepsilon_0 \,\,\mbox{ for all }n.
\end{equation}
where $\omega^n$ is the unique weak solution to the vorticity equation with initial vorticity $\omega^n_0.$

Applying Lemma \ref{ufcu}, we deduce from \eqref{e01} that there exists some $\tau_0>0$, not depending on $n$, such that
\begin{equation}\label{e02}
\inf_{w\in \mathcal W_p}\|w-\omega^n(t_n,\cdot)\|_{L^{1+\frac{1}{p}}(D)}\geq\tau_0 \,\,\mbox{ for each }n.
\end{equation}
By \eqref{e02} and the fact that $\mathcal W_p$ is orbitally stable in the $L^{1+\frac{1}{p}}$ norm of the vorticity with initial perturbations in $L^\infty(D)$ (this has been proved in Theorem \ref{thm1} by choosing $s=1+{1}/{p}$), we deduce that there exist a positive real number $\tau_1$ and a positive integer $N$ such that
\begin{equation}\label{e03}
\inf_{w\in \mathcal W_p}\|w-\omega^n_0\|_{L^{1+\frac{1}{p}}(D)}\geq\tau_1 \,\,\mbox{ for all }n\geq N.
\end{equation}

On the other hand, since $\{\omega^n_0\}_{n=1}^{+\infty}\subset \mathcal S_p$, we can choose a subsequence $\{\omega^{n_j}_0\}_{j=1}^{+\infty}$ and some $\eta\in \mathcal S_p$ such that as $j\to+\infty$
\begin{equation}\label{wcccc}
\omega^{n_j}_0\rightharpoonup \eta \,\,\mbox{ in } L^{1+\frac{1}{p}}(D),
\end{equation}
and thus
\begin{equation}\label{u000}
\lim_{j\to+\infty}E(\omega^{n_j}_0)=E(\eta).
\end{equation}
We claim that $\{\omega_0^{n_j}\}_{j=1}^{+\infty}$ is a maximizing sequence of $E$ over $\mathcal S_p$, that is,
\begin{equation}\label{u00}
\lim_{j\to+\infty}E(\omega^{n_j}_0)=M_p.
\end{equation}
To prove this, first we choose a sequence $\{w_j\}_{j=1}^{+\infty}\subset\mathcal W_p$
such that
\begin{equation}\label{dman}
\|\omega^{n_j}_0-w_j\|_E<\frac{1}{n_j} \,\,\mbox{ for any }j.
\end{equation}
This is doable by \eqref{e00}. Note that  $\|w_j\|_E=(2E(w_j))^{1/2}=(2M_p)^{1/2}$ for any $j$. Next we calculate $|E(\omega^{n_j}_0)-M_p|$
as follows
\begin{equation}\label{g01}
\begin{split}
|E(\omega^{n_j}_0)-M_p|=&|E(\omega^{n_j}_0)-E(w_j)|\\
=&\frac{1}{2}\left|\|\omega^{n_j}_0\|^2_{E}-\|w_j\|^2_{E}\right|\\
\leq&\frac{1}{2}\left|(\|\omega^{n_j}_0\|_{E}+\|w_j\|_{E})(\|\omega^{n_j}_0\|_{E}-\|w_j\|_{E})\right|\\
\leq &\frac{1}{2}(\|\omega^{n_j}_0\|_{E}+\|w_j\|_{E})\|\omega^{n_j}_0-w_j\|_{E}\\
=&\sqrt{2}M_p^{1/2}\|\omega^{n_j}_0-w_j\|_E.
\end{split}
\end{equation}
Here we used the fact
\[\|w\|_E=(2E(w))^{1/2}\leq (2M_p)^{1/2}\quad\forall\,w\in\mathcal S_p.\]
Combining \eqref{dman} and \eqref{g01} we obtain \eqref{u00}.

From \eqref{wcccc} and \eqref{u00}, we can apply Lemma \ref{scmcs} to obtain
 $\eta\in\mathcal W_p$ and
\begin{equation}\label{finish}
\omega^{n_j}_0\to\eta\,\, \mbox{ in } L^{1+\frac{1}{p}}(D),
\end{equation}
which is a contradiction to \eqref{e03}.

\end{proof}

\section{Proofs of Theorems \ref{thm3} and \ref{thm4}}

First we prove a lemma that will be used in the proof of Theorem \ref{thm4}.
\begin{lemma}\label{eqnms}
Let $\mathcal A$ be a subset of $\mathcal W_p.$ Denote $\mathcal A^c= \mathcal W_p\setminus\mathcal A.$ Then the following three items are equivalent
\begin{itemize}
\item [(i)] $\inf_{v\in\mathcal A, w\in \mathcal A^c}\|v-w\|_{L^1(D)}>0$;
\item[(ii)]$\inf_{v\in\mathcal A, w\in \mathcal A^c}\|v-w\|_{L^s(D)}>0$ for any $s\in[1,+\infty)$;
\item[(iii)] $\inf_{v\in\mathcal A, w\in \mathcal A^c}\|v-w\|_{E}>0$.
\end{itemize}
\end{lemma}

\begin{proof}
Without loss of generality assume that $\mathcal A$ and $\mathcal A^c$ are both nonempty.

First we show that (i) and (ii) are equivalent. Let $s\in[1,+\infty)$ be fixed. Since $\mathcal W_p$ is a bounded subset of $L^\infty(D)$ (see Lemma \ref{301}) and $D$ is a domain of finite Lebesgue measure, we can apply H\"older's inequality to obtain
\begin{equation}\label{isff}
c_1\|v-w\|_{L^1(D)}\leq \|v-w\|_{L^s(D)}\leq c_2\|v-w\|_{L^1(D)},\quad\forall\,v,w\in \mathcal W_p,
\end{equation}
where $c_1,c_2$ are positive constants depending only on $\mathcal W_p, D$ and $s$. From \eqref{isff} we immediately see that (i) and (ii) are equivalent.

Now we show that (ii) and (iii) are equivalent. First suppose (iii) holds.  By $L^p$ estimate and Sobolev embedding theorem, it holds that
\[\|v-w\|_{L^2(D)}\geq  c_3\|v-w\|_{E},\quad\forall\,v,w\in \mathcal W_p,\]
where $c_3$ is a positive constant depending only on $D$. From this we deduce that (ii) holds for $s=2$. Taking into account  the fact that (i) and (ii) are equivalent, we see that (ii) actually holds for arbitrary $s\in[1,+\infty).$ Now suppose (ii) holds. If (iii) is false, then we can choose two sequences $\{v_n\}_{n=1}^{+\infty}\subset \mathcal A$, $\{w_n\}_{n=1}^{+\infty}\subset \mathcal A^c$ such that
\begin{equation}\label{av10}
\lim_{n\to+\infty}\|v_n-w_n\|_{E}=0.
\end{equation}
Since $\mathcal W_p$ is compact in $L^{2}(D)$ (see Lemma \ref{bcyx}),
there exist subsequences $\{v_{n_j}\}_{j=1}^{+\infty}$, $\{w_{n_j}\}_{j=1}^{+\infty}$ and $\eta_1,\eta_2\in\mathcal W_p$ such  that
\begin{equation}\label{mked}
\lim_{j\to+\infty}\|v_{n_j}-\eta_1\|_{L^{2}(D)}=0,\quad\lim_{j\to+\infty}\|w_{n_j}-\eta_2\|_{L^{2}(D)}=0,
\end{equation}
which implies
\begin{equation}\label{mked2}
\lim_{j\to+\infty}\|v_{n_j}-\eta_1\|_{E}=0,\quad\lim_{j\to+\infty}\|w_{n_j}-\eta_2\|_{E}=0.
\end{equation}
From \eqref{av10} and \eqref{mked2} we get
\begin{equation}\label{xdd1}
\eta_1=\eta_2.
\end{equation}
On the other hand,
by choosing $s=2$ in (ii) we have
\[\inf_{v\in\mathcal A, w\in \mathcal A^c}\|v-w\|_{L^{2}(D)}>0,\]
hence $\|\eta_1-\eta_2\|_{L^{2}(D)}>0$, which contradicts \eqref{xdd1}.

\end{proof}

Now we are ready to give the proof of Theorem \ref{thm3}.

\begin{proof}[Proof of Theorem \ref{thm3} ]

Denote
\[d:=\inf_{w\in\mathcal W_p\setminus \{\tilde\omega\}}\|w-\tilde\omega\|_{L^s(D)}>0.\]

Since we have proved the orbital stability of $\mathcal W_p$ in Theorem \ref{thm1}, we see that for any $\varepsilon\in (0,d/4)$, there exists $\delta>0$, such that for any $\omega_0\in L^\infty(D)$ satisfying
\begin{equation}\label{p41}
\|\omega_0-\tilde\omega\|_{L^s(D)}<\delta,
\end{equation}
 it holds that
  \begin{equation}\label{p42}
  \inf_{w\in \mathcal W_p}\|\omega(t,\cdot)-w\|_{L^s(D)}<\varepsilon \,\,\mbox{  for any } t>0,
 \end{equation}
 where $\omega(t, x)$ is the unique weak solution to the vorticity equation with initial vorticity $\omega_0$. Without loss of generality we assume that
 \begin{equation}\label{ded2}
 0<\delta<d/2.
 \end{equation}

Now we claim that
  \begin{equation}\label{p43}
\|\omega(t,\cdot)-\tilde\omega\|_{L^s(D)}<\varepsilon \,\,\mbox{  for any } t>0.
 \end{equation}
In fact, if this is not true, then there exists some $t_1>0$ such that
  \begin{equation}\label{p44}
\|\omega(t_1,\cdot)-\tilde\omega\|_{L^s(D)}\geq \varepsilon,
 \end{equation}
which together with \eqref{p42} implies that there exists some $w_1\in\mathcal W_p\setminus\{ \tilde\omega\}$, such that
  \begin{equation}\label{p45}
\|\omega(t_1,\cdot)-w_1\|_{L^s(D)}<\varepsilon.
 \end{equation}
From \eqref{p45} and the definition of $d$ we get
\begin{equation}\label{p46}
\|\omega(t_1,\cdot)-\tilde\omega\|_{L^s(D)}\geq \|\tilde\omega-w_1\|_{L^s(D)}-\|\omega(t_1,\cdot)-w_1\|_{L^s(D)}\geq d-\varepsilon\geq\frac{3}{4}d.
\end{equation}
Since $\omega\in C([0,+\infty);L^{s}(D))$ (see (i) in Yudovich's theorem), we deduce from \eqref{p41}, \eqref{ded2} and \eqref{p46} that there exists some $t_2$ such that
\begin{equation}\label{p47}
\|\omega(t_2,\cdot)-\tilde\omega\|_{L^s(D)}=\frac{d}{2}.
\end{equation}
Thus for any $w\in\mathcal W_p\setminus{\{\tilde\omega\}}$, we get
  \begin{align}\label{p48}
\|\omega(t_2,\cdot)-w\|_{L^s(D)}\geq\|\tilde\omega-w\|_{L^s(D)}-\|\omega(t_2,\cdot)-\tilde\omega\|_{L^s(D)}\geq d-\frac{d}{2}=\frac{d}{2},
 \end{align}
which together with \eqref{p47} contradicts \eqref{p42}.  Hence \eqref{p43} is proved.

To summarize, we have proved that for any $\varepsilon\in (0,d/4)$, there exists some $\delta>0$, such that for any $\omega_0\in L^\infty(D)$ satisfying
\[\|\omega_0-\tilde\omega\|_{L^s(D)}<\delta,\] it holds that \[\|\omega(t,\cdot)-\tilde\omega\|_{L^s(D)}<\varepsilon\quad\forall\,t>0.\] Therefore the proof is finished.

\end{proof}

\begin{remark}
By Proposition \ref{prop1}, it is easy to check that an isolated least solution $\tilde\omega$ must be an isolated maximizer of $E$ over $\mathcal R_{\tilde\omega}.$ Thus Theorem \ref{thm3} can also be proved by applying  Burton's stability criterion in \cite{B5}. Either way, Proposition \ref{prop1} is the key point.
\end{remark}

The proof of Theorem \ref{thm4} is similar.
\begin{proof}[Proof of Theorem \ref{thm4}] First note that for any weak solution $\omega(t,x)$ to the vorticity equation, $\|\omega(t,\cdot)\|_E$ is a continuous  function of  $t$.
By Lemma \ref{eqnms}, for any isolated least energy solution $\tilde\omega\in\mathcal W_p$ it holds that
\[\inf_{w\in\mathcal W_p\setminus \{\tilde\omega\}}\|w-\tilde\omega\|_{E}>0.\]
Then the proof is almost identical to that of Theorem \ref{thm3} after replacing the $L^s$ norm $\|\cdot\|_{L^s(D)}$ by the energy norm $\|\cdot\|_{E}.$

\end{proof}

\begin{remark}
Based on the same argument, we can also prove the orbital stability of any isolated subset of $\mathcal W_p$.  More precisely, for a nonempty subset $\mathcal A$ of $\mathcal W_p$, if
\[\inf_{v\in \mathcal A, w\in\mathcal W_p\setminus \mathcal A}\|v-w\|_{L^1(D)}>0,\]
then $\mathcal A$ is orbitally stable in the $L^s$ norm of the vorticity with respect to initial perturbations in $L^\infty(D),$ where $s\in(1,+\infty),$ and is also orbitally stable in the energy norm $\|\cdot\|_E$ with respect to initial perturbations in $ \mathcal S_p\cap L^\infty(D).$
As a result, the set of positive least energy solutions $\{w\in\mathcal W_p\mid w>0 \mbox{ in } D\}$ and the set of negative least energy solutions $\{w\in\mathcal W_p\mid w<0 \mbox{ in } D\}$ are both orbitally stable in the above sense.

\end{remark}

\section{$0<p\leq 1$}

In the preceding sections, we have assumed that $p\in(1,+\infty)$. In this section, we briefly discuss the case of $p\in(0,1]$ in an analogous way.
For $p\in(0,1),$ we consider
\begin{equation}\label{p001}
\begin{cases}
-\Delta u=u^p &x\in D,\\
u>0&x\in D,\\
u=0&x\in\partial D.
\end{cases}
\end{equation}
For $p=1,$ we consider
\begin{equation}\label{p011}
\begin{cases}
-\Delta u=\lambda_1 u &x\in D,\\
u>0&x\in D,\\
u=0&x\in\partial D,\\
\|-\Delta u\|_{L^2(D)}=1.
\end{cases}
\end{equation}
Since  both \eqref{p001} and \eqref{p011} admit a unique solution (see below),  they are natural generalizations of the
case of $p\in(1,+\infty).$

\subsection{$0<p<1$}

For \eqref{p001},  there exists a unique solution  $\tilde u_p$ for any smooth bounded domain $D$ and $p\in(0,1)$. See  Theorem 7.41, p. 267 in \cite{Kra} for example.
Denote $\tilde\omega_p=-\Delta \tilde u_p$. Then $\tilde\omega_p$ is the unique function such that
\begin{equation}\label{sxxn}
\tilde\omega_p>0 \,\, \mbox{ a.e. in } D, \quad \tilde\omega_p=(\mathcal G\tilde\omega_p)^p.
\end{equation}

\begin{remark}
Note that $ \pm\tilde u_p $  are the only two maximizers of the natural energy functional $\mathcal I$ in $H^1_0(D).$

\end{remark}

\begin{proposition}\label{eccz}
Let $p\in(0,1).$ Then
$\tilde\omega_p$ and $-\tilde\omega_p$ are the only two  maximizers of $E$ over $\mathcal S_p,$ where $\mathcal S_p$ is defined by \eqref{dofsp}.
\end{proposition}
\begin{proof}
The proof is similar to that of Proposition \ref{prop1}, therefore we only sketch it. Repeating the arguments in Lemmas \ref{le222} and  \ref{profi} we can easily prove  that $E$ attains its maximum value over $\mathcal S_p$; moreover, for any maximizer $\bar \omega_p$, either $\bar\omega_p$ satisfies
\begin{equation}\label{prffo}
\bar\omega_p>0\,\, \mbox{a.e. in } D, \quad \|\bar\omega_p\|_{L^{1+\frac{1}{p}}(D)}=\|\tilde\omega_p\|_{L^{1+\frac{1}{p}}(D)},\quad \bar\omega_p=\lambda_p(\mathcal G\bar\omega_p)^p,
\end{equation}
or $-\bar\omega_p$ satisfies \eqref{prffo}, where $\lambda_p$ is a positive number depending on $\bar\omega_p.$ Then by the fact that \eqref{sxxn}  admits a unique solution $\tilde\omega_p$, we can further show that $\lambda_p=1$ and  \[\bar\omega_p=\tilde\omega_p\,\,\mbox{ or }\,\, \bar\omega_p=-\tilde\omega_p.\]

\end{proof}

Proceeding as in Section 3, the following stability theorem can be easily obtained by using the variational characterization proved in Proposition \ref{eccz}.
\begin{theorem}
$\tilde\omega_p$ is stable in the $L^s$ norm of the vorticity with respect to initial perturbations in $L^\infty(D)$, where $s\in(1,+\infty),$ and is also stable in the energy norm $\|\cdot\|_E$ with respect to initial perturbations in $\mathcal S_p\cap L^\infty(D)$.
\end{theorem}

\subsection{$p=1$}
It is well-known that for any smooth bounded domain $D$ \eqref{p011} admits a unique solution  $\tilde u_1$  (see \S6.5, \cite{E} for example).
Denote $\tilde\omega_1=-\Delta \tilde u_1$. Hence $\tilde\omega_1$ is the unique function satisfying
\begin{equation}\label{sxxn2}
\tilde\omega_1>0 \,\, \mbox{ a.e. in } D, \quad \tilde\omega_1=\lambda_1\mathcal G\tilde\omega_1,\quad \|\tilde\omega_1\|_{L^2(D)}=1.
\end{equation}

\begin{proposition}\label{thdwww}
$\tilde\omega_1$ and $-\tilde\omega_1$  are the only two maximizers of $E$ over
\[\mathcal H=\left\{w\in L^{2}(D)\mid \|w\|_{L^{2}(D)}\leq 1\right\}.\]
\end{proposition}
\begin{proof}
The proof can be proved in a similar way to that of Proposition \ref{prop1}.  Here we provide a more straightforward proof by using the method of eigenfunction expansion. Let $\{\phi_k\}_{k=1}^{+\infty}$ be the orthogonal basis of $L^2(D)$ satisfying
\[\phi_k\in H^1_0(D),\quad-\Delta \phi_k=\lambda_k \phi_k,\quad\|\phi_k\|_{L^2(D)}=1,\]
where $\{\lambda_k\}_{k=1}^{+\infty}$ is the set of eigenvalues of $-\Delta$ in $D$ with zero Dirichlet data, satisfying
\begin{equation}\label{increa}
0<\lambda_1<\lambda_2\leq\lambda_3\leq\cdot\cdot\cdot.
\end{equation}
It is easy to check that
\begin{equation}\label{ness4}
\phi_1=\tilde\omega_1\,\mbox { or }\,\phi_1=-\tilde\omega_1.
\end{equation}
For any $w\in\mathcal H$, we expand it in $L^2(D)$ in terms of  the basis $\{\phi_k\}_{k=1}^{+\infty}$ as follows
\[w=\sum_{k=1}^{+\infty}a_k\phi_k\]
Obviously
\begin{equation}\label{el11}
\sum_{k=1}^{+\infty}a^2_k=\|w\|^2_{L^2(D)}\leq1.
\end{equation}
Moreover, it is easy to check that
\[E(w)=\frac{1}{2}\int_Dw\mathcal G w dx=\frac{1}{2}\sum_{k=1}^{+\infty}\frac{a^2_k}{\lambda_k}.\]
Hence by \eqref{increa} and \eqref{el11} we have
\begin{equation}\label{0011}
2\lambda_1E(w)=\sum_{k=1}^{+\infty}\frac{\lambda_1}{\lambda_k}a_k^2\leq \sum_{k=1}^{+\infty}\frac{\lambda_1}{\lambda_1}a_k^2= 1.
\end{equation}
Moreover, it is easy to see that the inequality in \eqref{0011} is an equality if and only if $|a_1|=1$, $a_k=0, \forall\,k\geq 2.$ This means that $\phi_1$ and $-\phi_1$ are the only two maximizers of $E$ over $\mathcal H$, or equivalently by \eqref{ness4}, $\tilde\omega_1$ and $-\tilde\omega_1$ are the only two maximizers of $E$ over $\mathcal H$.
\end{proof}

As a consequence of Proposition \ref{thdwww}, we have

\begin{theorem}\label{5666}
$\tilde\omega_1$ is stable in the $L^s$ norm of the vorticity with respect to initial perturbations in $L^\infty(D)$, where $s\in(1,+\infty),$ and is also stable in the energy norm $\|\cdot\|_E$ with respect to initial perturbations in $ \mathcal H\cap L^\infty(D)$.

\end{theorem}

\begin{remark}
As far as we know, for the stability of $\tilde\omega_1$ in the energy norm, related results are few. It has been proved in \cite{LZ2} that there exists some domain $D$ such that the steady flow with vorticity $\tilde\omega_1$ is linearly unstable in the energy norm. However, whether the flow is nonlinearly stable in the energy norm is unknown. Theorem \ref{5666} provides a partial solution to this problem, indicating that for the perturbation class $ \mathcal H\cap L^\infty(D)$ the answer is positive.

\end{remark}

\bigskip

{\bf Acknowledgements:}
The author  would like to thank the anonymous referee for the valuable comments and suggestions on this papar.
 G. Wang was supported by National Natural Science Foundation of China (12001135, 12071098) and China Postdoctoral Science Foundation (2019M661261, 2021T140163).

\phantom{s}
 \thispagestyle{empty}

\end{document}